\documentclass[12pt,leqno]{amsart}

\usepackage{amssymb}
\usepackage{amsmath}
\usepackage{amscd}
\usepackage[T1]{fontenc}
\usepackage{enumitem}
\usepackage{amsmath}
\usepackage{amssymb,url,xspace}
\usepackage{mathrsfs,euscript,color}
\usepackage{xcolor}
\usepackage{bbm}
\usepackage{bm}

\newtheorem{proposition}{Proposition}[section]
\newtheorem{definition}[proposition]{Definition}
\newtheorem{lemma}[proposition]{Lemma}
\newtheorem{theorem}[proposition]{Theorem}

\newtheorem{conjecture}[proposition]{Conjecture}

\catcode`\Ž=13
\catcode`\=13
\catcode`\ˆ=13
\catcode`\=13
\catcode`\=13
\catcode`\‰=13
\catcode`\™=13
\catcode`\"=13
\catcode`\•=13
\catcode`\ž=13
\catcode`\=13
\catcode`\'=13
\catcode`\Ï=13

\def Ž{\'e}
\def {\`e}
\def ˆ{\`a}
\def {\`u}
\def {\^e}
\def ‰{\^a}
\def ™{\^o}
\def "{\^{\i}}
\def •{\"{\i}}
\def ž{\^u}
\def {\c c}
\def '{\"e}
\def Ï{\oe}

\def\mun{^{-1}}

\def\bsl{\backslash}
\def\bs{\boldsymbol}
\def\pni{\par\noindent}
\def\ptf{~.}
\def\vfe{\vfill\eject}

\def\com#1{\quad\hbox{#1}\quad}

\newcount\Cor\Cor=0
\def\newcor{\global\advance\Cor by 1
\par\bigskip\noindent (\romannumeral\Cor) - }

\def\FM{\mathbb F}
\def\GM{\mathbb G}

\def\ctyc{\mathcal{C}_c^\infty}

\def\coker{\mathrm{Coker}}

\def\vol{\mathrm{vol}\,}

\def\Fbar{\overline{F}}

\def\ve{\varepsilon}

\def\G{G}
\def\H{H}

\def\T{T}
\def\U{U}
\def\Z{Z}

\def\tC{{\widetilde C}}
\def\tG{{\widetilde G}}
\def\tI{{\widetilde I}}

\def\tT{{\widetilde T}}
\def\tZ{{\widetilde Z}}

\def\tf{\widetilde{f}}

\def\Q{Q}
\def\KK{\mathcal K}
\def\ECell{\bs\EC_{ell}}
\def\EC{\mathcal{E}}
\def\UC{\mathcal{U}^\G}
\def\SUC{\mathcal {SU}}

\def\SU{\mathcal{SU}}
\def\SGamma{\mathcal{S}\Gamma}

\def\x{x}
\def\g{g}
\def\tg{{\widetilde\g}}

\def\vef{{\ve_{E/F}}}
\def\tTef{\tT_{E/F}}
\def\Tef{T_{E/F}}
\def\Nef{N_{E/F}}

\def\orb{\mathcal{O}}
\def\Sorb{\mathcal{SO}}

\def\bu{{u}}
\def\bsu{\bs{u}}
\def\bstu{\bs{\tilde u}}
\def\bnu{{u_0}}
\def\bsnu{\bs{u_0}}
\def\bstnu{\bs{\tilde u_0}}
\def\iz{z}
\def\cft{c}

\date{\today}
\author{Jean-Pierre Labesse}
\address{Aix-Marseille UniversitŽ, CNRS, I2M, Marseille, France}

\title{Germ expansion for $\bs {SL(2)}$ in arbitrary  characteristics}

\begin{document}


\begin{abstract} Let $F$ be a  local field  of characteristic $p$ and $G$ be a connected reductive group over $F$.
Recall that Shalika's germ expansion \cite{Sh} of orbital integrals of regular semisimple elements
near the identity, when it exists, is a sum indexed by
the set $\UC$ of unipotent conjugacy classes in $G(F)$. Observe that if $G=SL(2)$ the set $\UC$
is always  compact; it is finite if $p\ne2$ while it is uncountable if $p= 2$. As a consequence,
Shalika's germ expansion for elliptic elements does not make sense 
if $p=2$. On the other hand the endoscopic expansion of elliptic orbital integrals always exists and yields 
a germ expansion equivalent if $p\ne2$ (up to a Fourier transform) 
to Shalika's germ expansion but is new if $p=2$.
\end{abstract}

\maketitle
\section{Introduction}
The asymptotic behaviour of orbital integrals of semisimple elements in a connected reductive group over $F$
is controlled by Shalika's germs when $F$ is a local field of characteritic zero and more generally whenever 
the number of unipotent orbits is finite and no inseparability occurs. The simplest case where these 
hypothesis fail is $SL(2,F)$ with $F$ of characteristic 2 since, for such a group, there is an uncountable
set of unipotent orbits.

Shalika's germ expansion is a very useful tool in local harmonic analysis
and in particular one would need a substitute for them to extend, in positive characteristic, the proof
by Kottwitz of Weil's conjecture on Tamagawa numbers for global fields of characteristic zero. 
Observe that the case of quasi-split groups
can readily be treated using Langlands-Lai technique since the spectral decomposition is available,
thanks to Morris, in arbitrary characteristic. But inner forms are, for the time being out of reach,
via Kottwitz' techniques, in small characteristics. We understand there is a proof by completely different
methods for positive characteristics but one would like to have a uniform approach independent of 
the characteristic.

This is our main motivation for looking more carefully at a very simple special case. It turns out that
a natural substitute valid in any characteristic can be given, for  $SL(2)$, using endoscopy. 
This appeared, with detailed proofs, in the lecture notes \cite{L1} where local and global 
aspects of the contribution of unipotent elements to harmonic analysis for $SL(2)$
in arbitrary characteristics are discussed. It seems reasonable to expect that
similar techniques will provide, for arbitrary reductive groups, the substitue one is looking for. 
This will be stated as a conjecture. But the state
of art, in (small with respect to the order of the Weyl group)
positive characteristic, does not allow to go much farther for the time being.

\section{Notation}
From now on, except in section \ref{conj}, let $\G=SL(2)$ and
$\tG=GL(2)$. We denote by $\Z$ the center of $\G$ and by $\tZ$ the center of $\tG$.
Then $\tG$ act by automorphisms, induced by conjugation, on $\G$.

Let $F$ be a non-archimedean local field of charateristic $p$.
The locally compact groups $\G(F)$ and $\tG(F)$ are endowed with  Haar measures denoted $d\g$ 
and $d\tg$ respectively.
Now let $f\in\ctyc(\G(F))$ and $\x\in\G(F)$.
We denote by $I_x$  the centralizer of $x$ in $\G$ (as a scheme) and by $\tI_x$ its centralizer in $\tG$.
Observe that $I_x(F)$ is always unimodular. This allows to define orbital integrals:
$$\orb_\G(x,f)=\int_{I_\x(F)\bsl\G(F)}f(\g\mun x\g)d\dot g\com{and}
\orb_\tG(x,f)=\int_{\tI_\x(F)\bsl\tG(F)}f(\tg\mun x\tg)d\dot \tg$$
where $d\dot \g$ and $d\dot \tg$ are quotient measures.

The quotient group $$Q_F:= (F^\times)^2\bsl F^\times$$ 
 is always compact. It is  finite if $p\ne2$ but  is uncountable when $p=2$.
 The determinant induces an isomorphism
$$\tZ(F)G(F)\bsl\tG(F)\simeq Q_F\ptf$$ 
The set $\UC$ of conjugacy classes of non trivial unipotent elements in $\G(F)$
is easily seen to be also in (non canonical) bijection with $Q_F$.

We shall denote by $\KK$ the Pontryagin dual of $\Q_F$. Its elements are of order 2 and 
Class Field theory shows that they are in bijection
with isomorphism classes of separable quadratic extensions $E/F$. The trivial character corresponds to
the split extension $E=F\oplus F$.

Let $E/F$ be a separable quadratic extension (that may be split). 
We denote by $e\mapsto\overline e$ the non trivial Galois automorphism (when $E=F\oplus F$
it acts by the exchange of coordinates).
Let $E^1$ denote the subgroup of elements of norm
1 in $E^\times$. We denote by $\tTef$ the restriction of scalar
for $E/F$ of $\GM_m$ and by $\Tef$ the subgroup of elements of norm 1. In particular we have
$$\tTef(F)=E^\times\com{and}\Tef(F)=E^1\ptf$$
We denote by $\vef$ the character of $F^\times$ associated to the quadratic extension $E/F$
by Class Field theory. If $E$  is a field
$\vef$ is the non trivial character of
$\Nef E^\times\bsl F^\times$; it is of order 2, while
$\vef$ is the trivial character if $E=F\oplus F$.

Let $T$ be a (maximal) torus in $\G$ and $\tT$ its centralizer in $\tG$.
There is a separable quadratic extension $E/F$  such that $\T$
(resp. $\tT$) is isomorphic to $\Tef$ (resp $\tTef$).  Conversely for any separable quadratic extension
$E/F$ there is a torus $T\subset\G$ such that $\T$ is isomorphic to
$\Tef$. One should be warned that, given $E$, there may be more than one conjugacy class of  tori in $\G$
isomorphic to $\Tef$.



\section{Stable conjugacy}

Two elements in $\G(F)$ (resp. $\tG(F)$)
will be said to be {\it stably conjugate} if they are conjugate in $\G(\Fbar)$ (resp. $\tG(\Fbar)$).
Observe that, in the literature, stable conjugacy is only defined for semisimple 
elements and even for them it is slightly more subtle for non strongly regular elements in arbitrary reductive group
(see for example  \cite[section I.4.4]{MW}).
But this simple definition works nicely for all elements 
in groups under consideration here.

The set $C(x)$ of conjugacy orbits under $\G(F)$ inside the set
of  rational points in the geometric orbit is parametrized by
 $$C(x)\simeq\coker[H^0_{f}(F,\G)\to H^0_{f}(F,I_x\bsl\G)]
\simeq\ker[H^1_{f}(F,I_x)\to H^1_{f}(F,\G)]$$
where the $H^i_f$ are flat cohomology sets. We have to use flat topology since, in characteristic 2,
the scheme $I_x$ has an infinitesimal factor when $x=zu$ with $z\in Z(F)$ and $u$ 
a non trivial unipotent. 
For all other elements one could use Žtale
cohomology as well since Žtale and flat cohomology coincide 
for smooth connected groups and one may forget about the index $f$. 
Since $H^1_{f}(F,\G)$ is trivial we have
$$C(x)\simeq H^1_{f}(F,I_x)\ptf$$
Then standard cohomology results show that $C(x)$ is trivial if $x$ belongs to a split torus, it has two elements 
if $x$ is regular semisimple in a non split torus and finally
$$C(x)\simeq Q_F$$ when $x$ has a non-trivial unipotent factor in his Jordan decomposition.
In particular it is an uncountable set when $F$ is of characteristic 2.
We also observe that if $T$ is a torus in $\G$
$$H^0(F,\T\bsl\G)=\tT(F)\bsl\tG(F)\ptf$$

\begin{lemma} \label{tg}
In $\G(F)$ (resp. $\tG(F)$) two elements are stably conjugate if and only if they are
conjugate under $\tG(F)$.
\end{lemma}

\begin{proof}
Firstly one observes that since $\G(\Fbar).\tZ(\Fbar)=\tG(\Fbar)$
conjugacy under $\G(\Fbar)$ and $\tG(\Fbar)$ coincide.
Then it remains to show that conjugacy and stable conjugacy coincide
in $\tG(F)$.  Observe that $\tI_x$ 
is smooth and connected for all $x\in\tG(F)$. In fact 
$\tI_x=Res_{E/F}\GM_m$ for some separable quadratic extension
if $x$ is regular semisimple, $\tI_x=\tG$ if $x$ is central while
 $\tI_x=\GM_m\times\GM_a$ if the Jordan decomposition of $x$ has a non trivial unipotent component.
The set $\tC(x)$ of conjugacy orbits under $\tG(F)$ inside the set
of  rational points in the geometric orbit is parametrized by
 $$\tC(x)\simeq\coker[H^0_{f}(F,\tG)\to H^0_{f}(F,I_x\bsl\tG)]
\simeq\ker[H^1_{f}(F,\tI_x)\to H^1_{f}(F,\tG)]\ptf$$
It remain to observe, that  $H^1_{f}(F,\tI_x)$ is trivial for any $x\in\tG(F)$. 
\end{proof}

\section{Endoscopic expansion}

For our $\G$ an endoscopic character $\kappa$ is an element of $\KK$
(the Pontryagin dual of $\Q_F$)
viewed as a character of $\tG(F)$ thanks to the isomorphism
$$\tZ(F)G(F)\bsl\tG(F)\simeq Q_F= (F^\times)^2\bsl F^\times$$ 
induced by the determinant.
This allows to define $\kappa$-orbital integrals by
$$\orb_\G^\kappa(\x,f)=
\int_{\Q_F}\kappa(q)\orb_\G(x^q,f)\, dq
=\int_{\tZ(F)I_x(F)\bsl\tG(F)}\kappa(\det \tg)f(\tg\mun \x  \tg)\,d\dot\tg$$
where, by abuse of notation, $x^q=\tg\mun\x\tg$
and $\tg$ is any element in $\tG(F)$ that maps to $q$ via the determinant.
This makes sense since $\orb_\G(x^q,f)$ is independent of the choice of $\tg$.
The compact group $Q_F$ is endowed with the canonical
Haar measure: $\vol(Q_F)=1$. 

When $\kappa=1$ integral $\orb_\G^1(x,f)$ will also be denoted $\Sorb_\G(x,f)$ 
and is called the stable orbital integral. For $\kappa\not=1$ they are called unstable.
By Fourier inversion, we have
$$\orb_\G(\x,f)=\sum_{\kappa\in\KK}\orb_\G^\kappa(\x,f)\ptf$$
The series converges since all but a finite number, depending on $f$, of $\kappa$-orbital integrals
 $\orb_\G^\kappa(x,f)$ do vanish. Let us discuss the various cases. For $z\in\Z(F)$ we have
$$f(z)=\orb_\G(z,f)=\Sorb_\G(z,f)
\com{while}\orb_\G^\kappa(z,f)=0 \com{for all} \kappa\not=1\ptf$$
Now consider a torus  $\T$ in $\G$ and $t\in T(F)$ regular.
Since $$\tg\mapsto f(\tg\mun t \tg)$$ is  left-invariant under
 $\tT(F)$, the centralizer of $t$ in $\tG(F)$,
 the $\kappa$-orbital integral $\orb_\G^\kappa(t,f)$ vanishes unless
$\kappa$ is trivial on $\tT(F)$ which implies
 $\kappa=1$ or $\kappa=\vef$ where $E/F$ is the quadratic extension 
 defined by $\T$. In such a case we have\footnote{We observe that the open inclusion
$$\T(F)\bsl\G(F)\subset\tT(F)\bsl\tG(F)$$ 
defines an invariant measure on $\tT(F)\bsl\tG(F)$
which may not be the measure used to define $\orb_\G^\kappa(t,f)$.}
 $$\orb_\G^\kappa(t,f)=\int_{H^0(F,\T\bsl\G)}\kappa(\det \tg)f(\tg\mun t \tg)d\dot\tg=
\int_{\tT(F)\bsl\tG(F)}\kappa(\det \tg)f(\tg\mun t \tg)d\dot\tg\ptf$$
 If $E$ is a field one has $$\orb_\G^\vef(t,f)=\frac{1}{2}\big(\orb_\G(t,f)-\orb_\G(t',f)\big)$$
where $t'$ is stably conjugate but non conjugate to $t$,
i.e. $t'=\tg\mun t\tg$ with  $\tg\in\tG(F)$ and $\vef(\det \tg)=-1$.
When $E=F\oplus F$ we have $\vef=1$ and  
$$\orb_\G(t,f)=\Sorb_\G(t,f)=\orb_\G^{\vef}(t,f)\ptf$$
Finally, when $x=zu$ with $z\in\Z(F)$ and $u$ a non trivial unipotent there is no {\it a priori}
vanishing for $\orb_\G^\kappa(x,f)$.

For the group $\tG$, lemma \ref{tg} shows that
orbital integrals are automatically stable. The same is trivially true for tori 
viewed as reductive groups since conjugacy classes are singletons. 

\section{Germ expansion for $p\not=2$}

Consider the unipotent element
$${\bnu}:=\begin{pmatrix}1&1\cr0&1\cr\end{pmatrix}\ptf$$
The map $\nu:\U(F)\to F$:
$$ \begin{pmatrix}1&\eta\cr0&1\cr\end{pmatrix}\mapsto\eta$$
induces a bijection between the set $\UC$ of rational conjugacy classes of 
non trivial unipotent elements in $\G(F)$
and the compact group $\Q_F$ where $\bnu$ corresponds to $1\in\Q_F$.
Assume now $p\not=2$; then the set $\UC$ is finite
and one has Shalika's germ expansion when $t$ tends to $z\in\Z(F)$:
$$\orb_\G(t,f)= \Gamma^\G_1(t)f(z)+\sum_{\bsu\in\UC} \Gamma^\G_{\bsu}(t)\orb_\G(z u,f)$$
where we denote by $u$ a representative of the conjugacy class  $\bsu$.
 By Fourier inversion we see that, if we put $<\kappa, \bsu>:=\kappa(\nu(\bsu))$, then
$$\orb_\G(z\bu,f)=\sum_{\kappa\in\KK}<\kappa, \bsu>\orb_\G^\kappa(z{\bnu},f)\ptf$$
Shalika's germ expansion can be rewritten
$$\orb_\G(t,f)= \Gamma^\G_1(t)f(z)+
\sum_{\bsu\in\UC}\sum_{\kappa\in\KK}<\kappa, \bsu> 
\Gamma^\G_{\bsu}(t)\orb_\G^\kappa({\iz \bnu},f)\ptf$$
It is then natural to introduce  $\kappa$-germs:
$$\Gamma^{\G,\kappa}_{\iz \bnu}(t):=\sum_{\bsu\in\UC}<\kappa, \bsu> \Gamma^\G_{\bsu}(t)\ptf$$
This yields an expansion with $\kappa$-germs
$$\orb_\G(t,f)= \Gamma^\G_1(t)f(z)+\sum_{\kappa\in\KK}
 \Gamma^{\G,\kappa}_{\iz \bsnu}(t)\orb_\G^\kappa({\iz \bnu},f)\ptf$$
Combined with endoscopic transfer this yields another form of germ expansion 
that will be shown to make sense in arbitrary characteristic
for $\G=SL(2)$.

\section{Endoscopic transfer}

We now introduce objects for local endoscopy. An endoscopic pair $\EC$ is a pair $\{H,\kappa\}$
where $H$ is a quasi-split connected reductive group over $F$ said to be an endoscopic group and $\kappa$
is an endoscopic character. For $\G=SL(2)$ endoscopic pairs are of the form
$\{\G,1\}$ or $\{\Tef,\vef\}$ where $E/F$ is a separable quadratic extension.
Let $\EC= \{H,\kappa\}$ be an endoscopic pair and $\T$ a torus in $\G$.
Assume there is an embedding $$\iota:T(F)\to\H(F)\ptf$$
This is always possible if $\H=\G$ while, if $\H=\Tef$, this implies that $E/F$
is the quadratic extension defined by $T$. In such a case we denote by $\overline\iota$
the compositum of $\iota$ with the Galois automorphism,
so that $\iota(t)$ and $\overline\iota(t)$ can be seen as the eigenvalues of $t$.
Moreover we fix a regular element $\tau\in\T(F)$.  
We may now define transfer factors.
\begin{definition}\label{FTr} \label{ft}
For $\EC= \{\G,1\}$ the transfert factor is the scalar
$$\Delta^\EC(t)=1\ptf$$
For $\EC=\{\Tef,\vef\}$ and $\T$ isomorphic to $\Tef$ the transfert factor is the scalar
$$\Delta^\EC(t)=\cft\,\vef\Big(\frac{\iota(t)-\overline\iota(t)}{\iota(\tau)-\overline\iota(\tau)}\Big)
{\vert{\iota(t)-\overline\iota(t)}\vert_E}$$ 
where $\cft$ is a constant.
\end{definition}
 
We shall not discuss natural choices for the constant $c$ since this is irrelevant for what follows.
\begin{definition}\label{def}
Given an endoscopic pair $\EC=\{H,\kappa\}$,
a function $$f^\EC\in\ctyc(H(F))$$ will be called an endoscopic transfer of $f$ if,
given  a torus $T$ in $\G$ together with an injection $\iota:T\to H$ then, for any regular 
$t\in\T(F)$, there is an identity of the form
$$\Sorb_\H(\iota(t),f^\EC)=\Delta^\EC(t)\orb_\G^\kappa(t,f)\ptf$$
\end{definition}

\begin{theorem}\label{trans} Given $f\in\ctyc(\G(F))$ 
an endoscopic transfer $f^\EC$ always exists. Moreover
$$f^{\EC}(\overline\iota(t))=f^{\EC}({\iota(t)})\ptf$$
\end{theorem}

\begin{proof} The theorem is trivial when $\EC= \{\G,1\}$. When $\EC=\{\Tef,\vef\}$
a proof is given in \cite[Lemma 2.1]{LL}. It 
is an explicit computation based on an argument borrowed from the proof of  \cite[lemma 7.3.2]{JL}. 
One should observe that  \cite{LL} assumes that the characteristic of $F$ is zero, 
but this assumption is seldom used
and, in particular, plays no role in the proof of  \cite[Lemma 2.1]{LL}.
One can also find a detailed proof in \cite{L1}.
\end{proof}

One should notice that the transfer depends on various choices, namely:
of Haar measures, of  $\cft$ and  $\tau$.

\section{Germ expansion for $SL(2)$ in arbitrary characteristic}

Consider first stable orbital integrals $\Sorb_\G(t,f)$.
Let us denote by $\bstnu$ the orbit under $\tG(F)$ conjugacy of $\bnu$.

\begin{lemma}\label{stable}
The stable orbital integrals $\Sorb_\G(t,f)$ has a germ expansion when $t\in T(F)$ is regular and tends to $z\in\Z(F)$. 
Namely, if $t$ is regular and close enough to $z$, one has the identity
$$\Sorb_\G(t,f)= \SGamma_z^{\G}(t)f(z)+ \SGamma_{{\iz \bstnu}}^{\G}(t)\Sorb_\G({z\bnu},f)$$
where $\SGamma^\G_\star:=\Gamma^\tG_\star$  are  Shalika's germs for $\tG(F)$.
\end{lemma}
\begin{proof}
 We already observe that a stable orbital integral is nothing but
an orbital integral on $\tG(F)$: 
$$\Sorb_\G(t,f)=\orb_\tG(t,\tf)$$
where $\tf$ is any element in $\ctyc(\tG(F))$ such that $\tf|_{\G(F)}=f$ 
since, for $t\in\G(F)$, the restriction of $\tf$ to $\G(F)$ only matters for the computation of the orbital integral.
Assumptions of Shalika's theorem \cite{Sh}
are fulfilled for $\tG(F)$ in any characteristic
and hence stable orbital integrals for $\G$, viewed as orbital integrals for $\tG(F)$,
do have a germ expansion indexed by stable unipotent orbits in $\G(F)$, i.e.
usual unipotent orbits in $\tG(F)$.
\end{proof}
\begin{lemma}\label{unstable}
Let $\T$ be a torus isomorphic via $\iota$ to $\Tef$ and let $\EC=\{\Tef,\vef\}$. 
The $\kappa$-orbital integral $\orb_\G^\kappa(t,f)$ where $\kappa=\vef$,
has a germ expansion when $t\in T(F)$ is regular and tends to $z\in\Z(F)$. Namely,
if $t$ is regular and close enough to $z$ one has the identity
$$\orb_\G^\vef(t,f)=\Delta^\EC(t)\mun f^{\EC}(z)\ptf$$
\end{lemma} 
\begin{proof}This follows from \ref{trans}.
 \end{proof}
 
 These two lemmas together with the Fourier expansion
 $$\orb_\G(t,f)=\sum_{\kappa\in\KK}\orb_\G^\kappa(t,f)$$
 yield a germ expansion valid in arbitrary characteristics.
Let us make explict the various cases.
When $E$ is the quadratic field attached to an elliptic  torus $T$ and 
 $\EC=\{\Tef, \vef\}$ then, for $t\in\T(F)$ regular but close enough to $z\in\Z(F)$, one has
$$\orb_\G(t,f)= \SGamma_{\iz }^{\G}(t)f(z)+ \SGamma_{{\iz \bstnu}}^{\G}(t)\Sorb_\G({z\bnu},f)
+\Delta^\EC(t)\mun f^{\EC}(z)\ptf$$
When $\T$ is a split torus and $E=F\oplus F$ we have two ways for looking
at $\orb_\G(t,f)$ for $t\in\T(F)$
since, in such a case, $$ \orb_\G(t,f)= \Sorb_\G(t,f)=\orb_\G^\vef(t,f)\ptf$$
The split transfer yields readily an asymptotic expansion since,
 for $t$ close enough to $z\in\Z(F)$, one has
$$\orb_\G(t,f)=\Delta^\EC(t)\mun f^{\EC}({z})$$ 
where $f^{\EC}({z})$ is nothing but the stable orbital integral $\Sorb_\G(z\bnu,f)$. 
Hence for $\T$ split, $t\in\T(F)$ regular and close enough to $z\in\Z(F)$ one  also has
$$\orb_\G(t,f)=\Delta^\EC(t)\mun \Sorb_\G(z\bnu,f)\ptf$$ 

We may now summarize the previous observations.
We say that $\EC=\{\H,\kappa\}$ is an elliptic data if $\EC=\{\G,1\}$ or  $\EC=\{\Tef, \vef\}$ 
where $\kappa=\vef$ is non trivial.
We denote by $\ECell$ the set of elliptic data. Elliptic data are in bijection with $\KK$.
Consider a torus $\T$ and $t$ semisimple regular in $\T(F)$. Recall we have choosen
an isomorphism $\iota$ if $\T$ is isomorphic to $\Tef$. Now let
$$\Phi^\EC(t)=\begin{cases}\Delta^\EC(t)\mun\Sorb_\H(\iota(t),f^{\EC})
&\quad\hbox{if there is an embedding of $T$ in $\H$}
\cr\quad 0&\quad\hbox{otherwise}
\end{cases}$$

\begin{theorem} \label{th}
When $t\in\T(F)$ is regular but close enough to $z\in\Z(F)$ 
the orbital integral $\orb_\G(t,f)$ is a sum over elliptic endoscopic data:
$$\orb_\G(t,f)= \sum_{\EC=\{\H,\kappa\}\in\ECell}
\Phi^\EC(t)\ptf \leqno(1)$$
The stable orbital integral $ \Sorb_\H(f^{\EC}(\iota(t)))$ has a germ expansion indexed by 
the set $\SUC^\H$ of
stable unipotent conjugacy classes
in $\H(F)$:
$$\Sorb_\H(f^\EC(\iota(t)))=\sum_{\bstu\in\SU^\H} \SGamma^\H_{\iota(\iz) \bu}(\iota(t))
\Sorb_\H(\iota(z)\bu,f^\EC)\leqno{(2)}
$$
where $\bu$ is a  representatives of  $\bstu$.
\end{theorem}

\begin{proof}We first recall that
$$\orb_\G(t,f)=\sum_{\kappa\in\KK}\orb_\G^\kappa(t,f)$$
and we conclude thanks to \ref{trans}, \ref{stable} and \ref{unstable}.
\end{proof}

\section{A conjecture}\label{conj}

In this section $\G$ is an arbitrary connected quasi-split reductive groupe over a non archimedean local field $F$. 
A naive conjecture would be that theorem \ref{th} is valid in general.
Already  assertion $(1)$ is known to hold at least for groups over local fields 
of characteristic zero and is expected to hold in general. 
But, as regards the second assertion, some care is in order.

Firstly, in positive characteristics, there may exist a primitive 
rational element $\gamma\in\G(F)$ (i.e. that do not belong to any non
trivial rational parabolic subgroup) that is conjugate over the algebraic closure to  $u\in\U(F)$
where $U$ is the unipotent radical of some parabolic subgroup $P$ defined over $F$. For example, 
if $F=\FM_2((X))$ and if we denote by $[M]$ the class modulo scalar matrices of a matrix $(M)$, then
$$\gamma=\begin{bmatrix}0&X\cr1&0\cr\end{bmatrix}
\com{is conjugate over $\FM_2((X^{1/2}))$ to}u=[\bnu]=\begin{bmatrix}1&1\cr0&1\cr\end{bmatrix}$$  
in $PGL(2,F)$. We have to consider a more restrictive notion of unipotent elements:
We call an  element  $\bu\in\G(F)$ true unipotent if it belongs to $U(F)$
where $U$ is the unipotent radical of some parabolic subgroup $P$ defined over $F$.  

An even more serious difficulty is that the vector space $\SUC_\G$ of stable distributions with true 
unipotent support is quite mysterious. 
 Already, for fields of characteristic zero, the vector space $\SUC_\G$ 
does not admits the set $\SUC^\G$, of conjugacy classes over the algebraic closure of 
rational unipotent elements, as a basis.
In fact, according to \cite{WS}, non special geometric orbits do not support any non trivial stable distribution while
special ones may support a space of dimension greater than 1 of stable distributions.

 Given a distribution $D\in \SUC_\G$ and a function $f\in\ctyc(\G(F))$
we denote by  $\langle D,f\rangle$ the value of $D$ on $f$. Let $\G_{reg}$ be the set of regular semisimple 
elements in $\G$.

\begin{conjecture} \label{conje}
Assume $\G$ is quasi-split. \pni
(1) The vector space $\SUC_\G$ is finite dimensional.\pni
(2) There is a map
$$\SGamma_\G:\G_{reg}(F)\to\SUC_\G$$
such that:
\pni (a)  the distribution $\SGamma_\G(t)$ depends only on the stable
conjugacy class of $t\in\G_{reg}(F)$,
\pni(b) for $f\in\ctyc(\G(F))$ and $\iz\in\Z(F)$ the center of $\G(F)$, there is a neighbourhood $V$ of $\iz$
 such that, for $t\in V\cap\G_{reg}(F)$, there is an identity:
$$\Sorb_\G(t,f)=\langle \SGamma_\G(t),f\rangle\ptf$$
\end{conjecture}

\bigskip

\section*{acknowledgements}
I must thank Waldspurger for pointing out a major mistake in an earlier formulation
of conjecture \ref{conje}. It is my pleasure to also thank Rapha\"el Beuzart-Plessis and Bertrand Lemaire for
useful comments.

\vfe

\end{document}